\documentclass{amsart}

\usepackage{appendix,amsfonts,amssymb,amsmath,enumerate,verbatim,mathtools,tikz,bm,mathrsfs,tikz-cd,hyperref,comment}
\hypersetup{
    colorlinks=true,
    linkcolor=blue,
    filecolor=magenta,      
    urlcolor=cyan,
}

\usepackage[utf8]{inputenc}


\title[Duality and symmetry of complexity]{Duality and symmetry of  complexity over  complete intersections via exterior homology}

\author[Jian Liu]{Jian Liu}
\address{School of Mathematical Sciences, University of Science and Technology of China, Hefei 230026, Anhui, P.R. China. }
\email{liuj231@mail.ustc.edu.cn}
\thanks{The first author thanks the China Scholarship Council for  financial support to visit Srikanth Iyengar at the University of Utah.}

\author[Josh Pollitz]{Josh Pollitz}
\address{Department of Mathematics,
University of Utah, Salt Lake City, UT 84112, U.S.A.}
\email{pollitz@math.utah.edu}
\thanks{The second author was supported by the National Science Foundation under Grant No. 1840190.}

\renewcommand{\S}{{\mathcal{S}}}

\DeclareMathOperator{\h}{H}

\newcommand{\T}{\mathsf{T}}

\newcommand{\D}{\mathsf{D}}

\newcommand{\vp}{\varphi}

\pgfdeclarelayer{bg}    
\pgfsetlayers{bg,main} 
\newcommand{\x}{{\bm{x}}}

\newcommand{\del}{\partial}

\newcommand{\m}{\mathfrak{m}}

\newcommand{\p}{\mathfrak{p}}

\DeclareMathOperator{\id}{id}

\DeclareMathOperator{\cx}{cx}

\DeclareMathOperator{\Spec}{Spec}

\DeclareMathOperator{\Hom}{Hom}
\DeclareMathOperator{\Ext}{Ext}

\DeclareMathOperator{\V}{V}

\DeclareMathOperator{\thick}{\mathsf{thick}}
\DeclareMathOperator{\Tor}{Tor}
\DeclareMathOperator{\Kos}{Kos}

\DeclareMathOperator{\supp}{\mathsf{supp}}

\newcommand{\shift}{{\mathsf{\Sigma}}}
\DeclareMathOperator{\RHom}{\mathsf{RHom}}

\newcommand{\ot}{\otimes^{\mathsf{L}}}

\newcommand{\xra}{\xrightarrow}

\newtheorem{theorem}{Theorem}[section]

\newtheorem{proposition}[theorem]{Proposition}

\newtheorem{lemma}[theorem]{Lemma}
\newtheorem{corollary}[theorem]{Corollary}

\theoremstyle{definition}

\newtheorem{remark}[theorem]{Remark}

\newtheorem{chunk}[theorem]{}

\newtheorem*{ack}{Acknowledgements}

\newtheorem{Thm}{Theorem}

\begin{document}

\maketitle

\begin{abstract}
    We study  homological properties of a locally complete intersection ring by importing facts from homological algebra over  exterior algebras. One  application is showing that  the thick subcategories of the bounded derived category of a locally complete intersection ring are self-dual under Grothendieck duality. This was proved by Stevenson when the ring is a quotient of a regular ring modulo a regular sequence; we offer two independent proofs in the more general setting. Second, we use these techniques to supply new proofs that  complete intersections possess symmetry of complexity.
\end{abstract}

\section*{Introduction}

Homological algebra over complete intersections is profoundly linked to the homological algebra over exterior algebras.
This was clarified in  \cite{AI3} where Avramov and Iyengar established a process to obtain homological information over complete intersections from the corresponding results over graded Hopf algebras.  Their techniques provided new, easier proofs of many known  results over  complete intersections.

For example, the complexity of a module measures the polynomial rate of growth of its Betti numbers while the injective complexity measures  the polynomial rate of growth of the module's Bass numbers. Using the process described above, Avramov and Iyengar easily deduced that over  complete intersections the complexity of a module agrees with its injective complexity, and both of these values are bounded above by the complexity of the residue field; they also employ their methods to show the latter   is exactly  the codimension of the complete intersection.

In this article, we adopt techniques from  \cite{AI3} to acquire further information about complete intersections. For the rest of the introduction $R$ is a commutative noetherian ring. The first main result is framed in terms of the derived category of $R$, denoted $\D(R)$.

We let  $\D^f(R)$ denote the full subcategory   of $\D(R)$ consisting of those complexes of $R$-modules whose total homology is  finitely generated. It inherits the structure of a triangulated category from $\D(R)$.  Recently, there has been much interest in understanding  the structure of thick subcategories of $\D^f(R)$, see for example  \cite{ABIM,CI,DGI,Letz,Pol,St,Tak}. Our first main result is the following:
\begin{Thm}\label{t1}If $R$ is locally complete intersection, then
each thick subcategory of $\D^f(R)$ is self-dual under Grothendieck duality. That is, for any thick subcategory $\T$ of $\D^f(R)$ and object $M$ in $\T$,  $\RHom_R(M,R)$ is in $\T$, as well. 
\end{Thm}
 Stevenson in \cite[4.11]{St} proved the result under the additional assumption that $R$ is a quotient of a regular ring modulo a regular sequence. One can also deduce Theorem \ref{t1} from \cite[4.11]{St} in conjunction with recent results of  Letz \cite[3.12 \& 4.5]{Letz}; details are provided in Remark \ref{rs}.

In this article, we present two proofs of Theorem \ref{t1} both of which rely  on a  local-to-global principle of Benson, Iyengar and Krause (see \ref{lg}) and the structure of thick subcategories in the derived category of an exterior algebra over a field (cf. \cite{CI}). The first proof uses the theory of cohomological support discussed in Section \ref{sCI}. Namely, we show that the containment of thick subcategories is encoded in the  support varieties of Avramov and Buchweitz defined in \cite{AB} (see Theorem \ref{Th} for a precise statement). 

The second proof makes direct use of the graded Hopf algebra structure of the exterior algebra to show that thick subcategories  over an exterior algebra on generators of homological degree one are fixed by Grothendieck duality (see Theorem \ref{closed}). 
 Furthermore, it is worth noting that  the both of the proofs of Theorem \ref{cor1}  do not rely on the full classification of thick subcategories in \cite[8.8]{St2}; making the proofs here simpler even in the case that $R$ is a quotient of a regular ring modulo a regular sequence. 

As a consequence of Theorem \ref{t1} we   obtain asmyptotic information over complete intersections. For example, we recover a result of Avramov and Buchweitz \cite[6.3]{AB} that says the eventual vanishing of Ext is equivalent to the eventual vanishing of Tor over locally complete intersections (cf. Corollary \ref{cor}). Furthermore, in the local case we can use Theorem \ref{t1} to show complexity is symmetric in $M$ and $N$;
recall   the complexity of a pair of objects $M$ and $N$ of $\D^f(R)$ is the polynomial rate of growth of the minimal number of generators of $\Ext_R^n(M,N)$ (see \ref{cx} for a precise definition).
\begin{Thm}\label{t2}
If $R$ is  complete intersection, then $\cx_R(M,N)=\cx_R(N,M)$ for each pair of objects $M$ and $N$ in $\D^f(R).$
\end{Thm}

 This was first proven by Avramov and Buchweitz \cite{AB}; an alternative proof was provided by the second author in \cite{Po}. In contrast, we give two new proofs of this result  in this paper; both of which use the homological properties over exterior algebras. The first proof deduces Theorem \ref{t2} from Theorem \ref{t1}, illustrating how the containment  of thick subcategories  and duality can provide asymptotic information.  The second proof directly links the complexity of a pair of objects $M$ and $N$  in $\D^f(R)$ with the complexity of a pair  of objects in the derived category of an exterior algebra over a field. This  proof fills in a missing piece of the work in \cite{AI3}; cf. \cite[6.9]{AI3} and the discussion in Remark \ref{rcx}.

\begin{ack}
Both authors are indebted to Srikanth Iyengar for  many helpful discussions, as well as suggesting the two authors collaborate because of their many common interests. We are also very happy to thank Benjamin Briggs for several useful comments on an earlier draft of this paper as well as numerous conversations that helped  clarify some  ideas in Section \ref{a1}. We also thank Janina Letz and Greg Stevenson for their comments on a preliminary draft of this work.
\end{ack}

\section{Background, Notation, and Terminology} Throughout this article $R$ will be a commutative noetherian ring.

 \subsection{(Locally) Complete Intersections} 
 Suppose $(R,\m,k)$ is local. Recall the \emph{embedding dimension} of $R$ is $\dim_k\m/\m^2$, the minimal number of generators for $\m$, and the \emph{codimension of R} is 
\[
\dim_k\m/\m^2-\dim R.
\]

A  local ring $(R,\m,k)$ is  \emph{complete intersection} if its $\m$-adic completion $\widehat{R}$ is isomorphic to $Q/I$ where $Q$ is a regular local ring and  $I$  is generated by  a $Q$-regular sequence. In fact, the presentation can be chosen so that $Q$ and $R$ have the same embedding dimension  and $I$ is generated by $c$ elements where $c$ is the codimension of $R$.

More generally,  $R$ is \emph{locally complete intersection}  provided that the local ring $R_\p$ is  complete intersection for each prime ideal $\p$ of $R$.

\subsection{Derived Category of a DG Algebra}

Let $A$ be a DG $R$-algebra. We briefly discuss the derived category of DG $A$-modules and set notation used throughout the rest of the article. See \cite[Section 3]{ABIM} or \cite[Chapter 6]{FHT} for more details. 

Let  $\D(A)$ denote the derived category of (left) DG $A$-modules. Recall that $\D(A)$ is a triangulated category with $\shift$ being the suspension functor; for each $X$ in $\D(A)$, $\shift X$ is the DG $A$-module given by $\shift X_i=X_{i-1}$, $a\cdot (\shift x)=(-1)^{|a|}ax$ and $\del^{\shift X}=-\del^X$. We let $\D^f(A)$ denote the full subcategory of $\D(A)$ consisting of those objects $X$ of $\D(A)$ such that $\h(X)$ is a finitely generated graded $\h(A)$-module.

Each DG $A$-module $X$ admits a semiprojective resolution. That is, there exists a surjective quasi-isomorphism $P\to M$ such that $\Hom_A(P,-)$ preserves surjective quasi-isomorphisms. 
 For any $Y$ in $\D(A)$, we set  \[\RHom_A(X,Y)\coloneqq\Hom_A(P,Y)\] where $P\to X$ is a semiprojective resolution and  \[\Ext_A(X,Y)\coloneqq\h\left(\RHom_A(X,Y)\right),\] which naturally inherits  a graded $\Ext_A(Y,Y)$-$\Ext_A(X,X)$-bimodule structure.

\subsection{Koszul Complexes} 
\label{kos} Background on Koszul complexes can be found in \cite[Section 1.6]{BH}. We recall the necessary facts here. 

For a list of elements $\x=x_1,\ldots, x_n$ in $R$, we set $\Kos^R(\x)$ to be the Koszul complex of $\x$ on $R$, which is regarded as a DG $R$-algebra in the usual way.  

When $R$ is local with maximal ideal $\m$, set $K^R$ to be the Koszul complex on a minimal generating set for $\m$. It is well-defined up to an isomorphism of DG $R$-algebras. 

Fix  a prime ideal  $\p$ of $R$ and  let $M$ be an object of $\D(R).$ We set 
$$M(\p)\coloneqq M_\p\otimes_{R_\p} K^{R_\p}$$   which is a DG $K^{R_\p}$-module.
Restricting scalars along the morphism of DG algebras $R_\p \to K^{R_\p}$ we may regard $M(\p)$ as an object of $\D(R_\p)$.

 \subsection{Koszul Complexes over  Complete Intersections} \label{koszulci}
 Let $(R,\m,k)$ be  complete intersection of codimension $c$ and  $\Lambda$  be the  exterior algebra over $k$ on  generators $e_1,\ldots, e_c$ of homological degree $1$.   Let $\mathsf{t}\colon \D(R)\to \D(K^{R})$ be the functor  $-\otimes_R K^{R}$.
 
 By  \cite[6.4]{AI3} there is a quasi-isomorphism of DG algebras $K^R\simeq  \Lambda$ that
   induces an equivalence of triangulated categories $\mathsf{j}\colon\D(K^R)\to  \D(\Lambda)$;  this restricts to an equivalence $ \D^f(K^R)\xrightarrow{\equiv}\D^f(\Lambda)$ that is compatible with Grothendieck duality
(see \cite[3.6]{ABIM} or \cite[2.5]{AI3}). Hence, when $R$ is complete intersection we have the following composition
\[\mathsf{jt}:\D^f(R)\to \D^f(K^R)\xra{\equiv}\D^f(\Lambda);\] this is the main bridge for  importing results over graded exterior algebras to  complete intersections.
Throughout the rest of the paper, $\mathsf{j}$ and $\mathsf{t}$ will always denote the functors introduced here. 

\subsection{Thick Subcategories}\label{t}
Let $A$ be a DG algebra and $\T$ be a full subcategory of $\D(A)$. We say $\T$
 is \emph{thick} if it is a triangulated subcategory that is closed under taking direct summands. For an object $M$ of $\D(A)$, we let $\thick_{\D(A)} (M)$ denote the smallest thick subcategory of $\D(A)$ containing $M$.  This can be realized as the intersection of all thick subcategories  of $\D(A)$ containing $M$; alternatively, this has an inductive construction discussed in \cite[2.2.4]{ABIM}. 

\subsection{Local-to-Global Principle}\label{lg}

The main results in the present paper rely  on the following \emph{local-to-global principle} of Benson, Iyengar and Krause (see by \cite[5.10]{BIK2}).
Namely,  for objects $M$ and $N$ of $\D^f(R)$, $M$ is in $\thick_{\D(R)}(N)$ if and only if $M(\p)$ is in $\thick_{\D(R_\p)}\left(N_\p\right)$ for each prime ideal $\p$ of $R$.
As $N(\p)$ is an object of $\thick_{\D(R_\p)}\left( N_\p\right)$, we  restate the local-to-global principle as:
\[M\text{ is in }\thick_{\D(R)} (N) \iff   M(\p)\text{ is in }\thick_{\D(R_\p)}\left(N(\p)\right)\]
for each  prime ideal $\p$ of $R.$

\subsection{Homogeneous Support}
Let $\S$ be a commutative noetherian graded ring. We let $\Spec^* \S$ denote the homogeneous spectrum of $\S$. That is, $\Spec^*\S$ consists of the homogeneous prime ideals of $\S$. For a graded $\S$-module $X$ and $\p\in \Spec^*\S$, $X_\p$  denotes the homogeneous localization of $X$ at $\p$. The homogeneous support of $X$ is $$\supp_\S X=\{\p\in \Spec^*\S: X_\p\neq 0\}.$$


\section{Cohomological Support Varieties}
\label{sCI}

Throughout this section we fix the following notation. 
Let $(R,\m,k)$ be  complete intersection with  codimension $c$ and embedding dimension $\nu.$
Let $\S$ denote the graded $k$-algebra $k[\chi_1,\ldots,\chi_c]$  where each $\chi_i$ has homological degree $-2$.
 We set  $\Lambda$ to be the  exterior algebra over $k$ on  generators $e_1,\ldots, e_c$ of homological degree $1$. Finally, let $\mathsf{j}$ and $\mathsf{t}$ be the functors from \ref{koszulci}.

\begin{chunk}\label{act}
In \cite[Theorem 5]{Sj}, Sj{\"o}din described the graded $k$-algebra structure of $\Ext_R(k,k)$. It contains $\S$ as a polynomial subalgebra in  such a way that 
$$
\Ext_R(k,k)\cong \S\otimes_k \bigwedge \shift^{-1}k^{\nu}
$$ as graded $\S$-modules; see also \cite[10.2.3]{IFR} for more details. Thus,
 $\S$ acts on $\Ext_R(k,M)$  through the $\Ext_R(k,k)$-action for each  $M$  in $\D(R)$. We define  \emph{the cohomological support of $M$ over $R$} to be $$\V_R(M)\coloneqq\supp_\S \Ext_R(k,M).$$
\end{chunk}

\begin{chunk}\label{civar}
 By
 \cite[5.1]{AI3} (see also  \cite[7.4]{ABIM}), there is an isomorphism of graded $k$-algebras  $$
\Ext_\Lambda(k,k)\cong \S.$$
 For any $X$ in $\D(\Lambda)$, we define  the \emph{cohomological support of $X$  over $\Lambda$} to be $$\V_\Lambda(X)\coloneqq \supp_{\S}\Ext_\Lambda(k,X).$$
These varieties can detect  the containment of thick subcategories in $\D^f(\Lambda).$ Namely, in  \cite[4.4]{CI}, Carlson and Iyengar showed 
 \[
 X\text{ is in } \thick_{\D(\Lambda)} (Y)\iff \V_\Lambda(X)\subseteq \V_\Lambda(Y) 
 \]  for any pair of objects $X$ and $Y$ in $\D^f(\Lambda)$. This  essentially follows from  the celebrated theorem of Hopkins \cite[11]{H} and Neeman \cite[1.2]{N} (see  also  \cite[3.2]{CI} for the version needed) and a BGG correspondence (cf. \cite[7.4]{ABIM}).
\end{chunk}

There is a way to relate the supports defined over $R$ and $\Lambda.$ This was first noticed  in the case that $R$ is artinain \cite[5.11]{CI}; however, the same proof works without any restriction on the Krull-dimension of $R$ and so we sketch it for the convenience of the reader in the following remark and lemma. 
\begin{chunk}\label{injection} First, there is a canonical injective map  $\eta\colon \Ext_{\Lambda}(k,k)\to \Ext_R(k,k)$ of graded  $k$-algebras that can be factored as $$\Ext_{\Lambda}(k,k)\xra{\cong} \Ext_{K^R}(k,k)\hookrightarrow \Ext_R(k,k)$$ where the isomorphism is induced by the inverse of the equivalence $\mathsf{j}$ from \ref{koszulci}.   Moreover, the  image of $\eta$ is exactly the polynomial subalgebra $\S$ of $\Ext_R(k,k)$ mentioned in \ref{act}. Therefore, the   cohomological supports over $R$ and those over $\Lambda$ can naturally be thought of as subsets of the same  $\Spec^*\S$. Moreover, we have the following connection.
\end{chunk}

\begin{lemma}\label{supp2}
For each  $M$ in  $\D(R)$, 
$\V_R(M)=\V_R(\mathsf{t} M)=\V_\Lambda(\mathsf{jt} M).$
\end{lemma}
\begin{proof} First, consider the isomorphisms of graded $\S$-modules
\begin{align*}
\Ext_R(k,\mathsf{t}M)&\cong \Ext_{K^R}( \mathsf{t}k, \mathsf{t} M) \\
&\cong  \Ext_\Lambda(\mathsf{jt} k, \mathsf{jt}M) \\
&\cong \bigoplus_{i=0}^{\nu}\shift^{-i}\Ext_\Lambda(k,\mathsf{jt}M)^{{\nu \choose i}}
\end{align*}
where the third isomorphism holds because $\mathsf{t} k\cong\bigoplus_{i=0}^{\nu}\shift^{-i}k^{{\nu \choose i}} $ and $\mathsf {j}(k)\simeq k$, see \cite[3.9]{ABIM} for the latter.
Also, we have the isomorphism of graded $\S$-modules
$$\Ext_R(k,\mathsf{t}M)\cong \bigoplus_{i=0}^{\nu}\shift^i\Ext_R(k,M)^{{\nu \choose i}}.$$ Therefore, the isomorphisms of $\S$-modules show $$\V_R(M)=\V_R( \mathsf{t} M)=\V_\Lambda(\mathsf{jt}M).\qedhere$$ 
\end{proof}

\begin{proposition}\label{prop}
For  $M$ and $N$ in $\D^f(R)$, 
\[\mathsf{t} M\text{ is in }\thick_{\D(R)}(\mathsf{t} N)\iff \V_R(M)\subseteq \V_R(N).\]
\end{proposition}
\begin{proof} The forward direction is trivial from the first equality in Lemma \ref{supp2}. 

Conversely, assume $\V_R(M)\subseteq \V_R(N)$.
Using Lemma \ref{supp2}, this reads as  \[\V_\Lambda(\mathsf{jt} M)\subseteq \V_\Lambda(\mathsf{jt} N).\] Thus,  \ref{civar} implies that  $\mathsf{jt} M$ is an object of   $\thick_{\D(\Lambda)}  ( \mathsf{jt} N).$ As $\mathsf{j}$ is an equivalence we conclude that $\mathsf{t} M$ is an object of $\thick_{\D(K^R)}(\mathsf{t} N)$. The result follows by restricting scalars along the morphism of DG $R$-algebras $R\to K^R$. 
\end{proof}

We end this section with the following technical lemma which will be put to use in Section \ref{a1}. Note that since  $\Lambda$ has trivial differential, for each DG $\Lambda$-module $X$ we can negate differential of $X$ to obtain a DG $\Lambda$-module. Namely, let $X'$ denote the DG $\Lambda$-module whose underlying graded $\Lambda$-module is $X$ and its differential is $\del^{X'}\coloneqq -\del^X.$ When $\Lambda$ is concentrated in even degrees, $X\cong X'$ as DG $\Lambda$-modules (cf. \ref{twist}). However, as the generators of $\Lambda$ have degree 1 we do not know whether these are isomorphic. Instead, we show they have the same cohomological support, and hence, generate the same thick subcategory. 
\begin{lemma}\label{l:neg}
If $X$ is in $\D^f(\Lambda)$, then $\V_\Lambda(X)=\V_\Lambda(X').$ Moreover, we have the following equality of thick subcategories:
\[
\thick_{\D(\Lambda)}(X)=\thick_{\D(\Lambda)}(X').
\]
\end{lemma}
\begin{proof}
By \cite[4.2]{CI}, there exists a  semiprojective resolution $F\xra{\simeq} k$ over $\Lambda$ such that $F$ admits a DG $\S$-module structure compatible with the $\S$-action on $\Ext_R(k,Y)$ for any $Y$ in $\D(\Lambda).$ As $k$ has trivial differential, the same is true of the  semiprojective DG $\Lambda$-resolution equipped with a DG $\S$-module structure $F'\xra{\simeq} k$.

Define $\Phi: \Hom_\Lambda(F,X)\to \Hom_\Lambda(F',X')$ given by \[
\alpha\mapsto (-1)^{|\alpha|}\alpha.
\] As $\S$ is concentrated in even degrees this is an isomorphism of DG $\S$-modules. Therefore, $\h(\Phi)$ establishes the following isomorphism of graded $\S$-modules
\[
\Ext_\Lambda(k,X)\cong \Ext_\Lambda(k,X');
\]  so $X$ and $X'$ have the same cohomological support. The equality of thick subcategories now follows from \ref{civar}. 
\end{proof}


\section{Duality of Thick Subcategories  via Support}
\label{sd2}

In this section we give the first proof of our result on the duality of thick subcategories over  locally complete intersections (see Theorem \ref{cor1}). The  main idea behind it is that the theory of cohomological supports, discussed in  Section \ref{sCI}, both detects  containment of thick subcategories and is unaffected by duality. The first theorem addresses the former point while \ref{l2}  the latter. 
\begin{theorem}\label{Th}
Let  $R$ be locally complete intersection. For   $M, N$   in $\D^f(R)$,   \[M\text{ is in } \thick_{\D(R)}(N)\iff  
    \V_{R_\p}(M_\p)\subseteq \V_{R_\p}(N_\p)\] for each prime ideal $\p$ of $R$.
\end{theorem}
\begin{proof}
First, assume $M$ is an object of $\thick_{\D(R)}N$. Hence,   $M_\p$ is an object of  $\thick_{\D(R_\p)}N_\p$ and so it follows easily that  $\V_{R_\p}(M_\p)\subseteq \V_{R_\p}(N_\p)$ for $\p\in \Spec R.$

Conversely, suppose  $\V_{R_\p}(M_\p)\subseteq \V_{R_\p}(N_\p)$ for each $\p\in \Spec R.$  By Proposition \ref{prop}, $M(\p)$ is in $\thick_{\D(R_\p)} \left(N(\p)\right)$ for each prime ideal $\p$ of $R$. 
 Finally, we apply \ref{lg} to conclude that $M$ is in $\thick_{\D(R)}(N).$
 \end{proof}

\begin{chunk}\label{l2} It is well-known  that  cohomological  support over complete intersections  is closed under duality. That is, 
  \[
\V_R(M)=\V_R(\RHom_R(M,R))\] for each $M$  in $\D^f(R)$,
This was shown for closed points of $\Spec^*\S$  in \cite[3.3]{AB}, and  the general setting was shown in  \cite[4.1.5]{Po}.  Alternatively, a new proof is  obtained in the present work by combining  Lemma \ref{supp2} and  Theorem \ref{closed}; this proof sticks with the theme of establishing results for complete intersections by passing to an exterior algebra. 
\end{chunk}

\begin{theorem}\label{cor1}
If $R$ is locally complete intersection, then every thick subcategory of $\D^f(R)$ is closed under $\RHom_R(-,R)$. In particular, for each $M$ in $\D^f(R)$
\[
\thick_{\D(R)}(M)=\thick_{\D(R)}(\RHom_R(M,R)).
\]
\end{theorem}
\begin{proof}[First proof of Theorem \ref{cor1}] 
As $M$ is an object of  $\D^f(R)$, there is a natural isomorphism
$$
\RHom_R(M,R)_\p\cong \RHom_{R_\p}(M_\p,R_\p)
$$
for each prime ideal $\p$ of $R$. So   \ref{l2} shows
$$
\V_{R_\p}(M_\p)=\V_{R_\p}(\RHom_{R_\p}(M_\p,R_\p))=\V_{R_\p}(\RHom_R(M,R)_\p)
$$ for each prime ideal $\p$ of $R$.
Now we obtain \[\thick_{\D(R)}(M)=\thick_{\D(R)}(\RHom_R(M,R))\]as an immediate consequence of Theorem \ref{Th}.
\end{proof}

\begin{remark}\label{rs}
Theorem \ref{cor1} can also be proved by combining results of Stevenson and  Letz (see \cite[4.11]{St} and \cite[3.12 \& 4.5]{Letz}, respectively).  Stevenson used his classification of  thick subcategories of the singularity category of a regular ring modulo a regular sequence in \cite[8.8]{St2}  to show Theorem \ref{cor1} holds  for such rings. Letz showed for $M$ and $N$ in $\D^f(R)$, \[
M\text{ is in }\thick_{\D(R)}(N)\iff M\otimes \widehat{R_\p}\text{ is in }\thick_{\D(\widehat{R_\p})} (N\otimes_R \widehat{R_\p})\] for each prime ideal $\p$ of $R$ where $\widehat{R_\p}$ is the  $\p R_\p$-adic completion of $R_\p.$  So their work, indeed, offers a different argument for Theorem \ref{cor1}.  \end{remark}

 In Section \ref{a1}, we  give our second proof of Theorem \ref{cor1}. This requires an analysis of duality over a graded exterior algebra which is discussed there. We end this section with the following application  that recovers a theorem of Avramov and Buchweitz \cite[6.3]{AB}. In Section \ref{s:complexity}, we strengthen the equivalence of (1) and (2) in Corollary \ref{cor} when $R$ is further assumed to be local.
\begin{corollary}\label{cor}
Let $R$ be locally complete intersection. For $M$ and $N$ in $\D^f(R)$, the following are equivalent:
\begin{enumerate}
    \item $\Ext_R^i(M,N)=0$ for all $i\gg 0$;
    \item$\Ext_R^i(N,M)=0$ for all $i\gg 0$;
    \item$\Tor^R_i(M,N)=0$ for all $i\gg 0$.
\end{enumerate}
\end{corollary}
\begin{proof} It suffices to show the equivalence of (1) and (3) since Tor is symmetric in $M$ and $N$. Furthermore, the equivalence of (1) and (3) is the same as  
\[
\RHom_R(M,N) \text{ is in }\D^f(R)\iff M\ot_RN \text{ is in }\D^f(R);
\]so we show the latter equivalence holds. For the rest of the proof let $(-)^\vee$ be $\RHom_R(-,R).$ Observe for any pair of objects $M$ and $N$ in $\D(R),$ we have the following adjunction isomorphism:
\begin{equation}
\label{eeiso}\tag{$\dagger$}
(M\ot_RN)^\vee\simeq \RHom_R(M,N^\vee).
\end{equation}

$\impliedby:$ As $M\ot_R N$ is in $\D^f(R)$ and Theorem \ref{cor1},  it follows that $(M\ot_RN)^\vee$ is in $\D^f(R).$ Hence, by (\ref{eeiso}),
$\RHom_R(M,N^\vee)$ is in $\D^f(R)$. By Theorem \ref{cor1}, $N$ is in $\thick_{\D(R)}(N^\vee)$, and so   $\RHom_R(M,N)$ is in $\D^f(R),$ as well.

$\implies:$  
By Theorem \ref{cor1}, our assumption is equivalent to $\RHom_R(M,N^\vee)$ is in $\D^f(R). $ By (\ref{eeiso})
and  Theorem \ref{cor1} we conclude that $(M\ot_RN)^{\vee\vee}$ is in $\D^f(R)$, as well. 
Finally, as $R$ is Gorenstein, 
the natural map \[
M\to M^{\vee\vee}
\] is an isomorphism in $\D^f(R)$. Therefore, $M\ot_R N$ is in $\D^f(R),$ as needed.
\end{proof}


\section{Duality of Thick Subcategories via Exterior Algebras}\label{a1}
 
For a finite dimensional Hopf algebra $H$ over a field $k$ and finite dimensional $H$-module $M$, it is well known that $M$ is a direct summand of $M\otimes_k \Hom_k(M,k)\otimes_k M$ (see, for example, \cite[3.1.10]{benson});  the  $H$-module structure on the latter needs both the antipode and co-multiplication of $H$. In this section we make use of the DG version of this fact for a graded exterior algebra.

Throughout $k$ will be a field and  $\Lambda$ is the exterior algebra on  generators $e_1,\ldots,e_c$ of homological degree $1$ over $k$. It is easily checked $\Lambda$ is a graded Hopf algebra with co-multiplication and antipode  determined by 
$$
\Delta(e_i)=e_i\otimes 1+1\otimes e_i \
\text{ and } \ 
 \sigma(a)=(-1)^{|a|}a,
 $$
 respectively (cf. \cite[5.3]{AI3}).

The main goal of the   section is  to prove the next theorem and as an application we give a second short proof of Theorem \ref{cor1}. The proof of Theorem \ref{closed} requires some additional setup  and can be found after \ref{split}.

\begin{theorem}\label{closed}
For an object  $M$ of $\D^f(\Lambda)$, 
 \[\thick_{\D(\Lambda)} (M)=\thick_{\D(\Lambda)} (\RHom_\Lambda(M,\Lambda)).\]  
\end{theorem}

\begin{chunk}
\label{dual}
Let $\rho\colon \Lambda\to \Lambda$ be an anti-endomorphism of graded $k$-algebras, 
which is nothing more than an endomorphism  since $\Lambda$ is graded-commutative.  The map $\rho$ prescribes a natural left DG $\Lambda$-module structure on the graded $k$-space $M^*=\Hom_k(M,k)$;  define $M^*(\rho)$ to be the left DG $\Lambda$-module whose underlying graded $k$-space is $M^*$ and its differential and $\Lambda$-action are  given by \begin{align*} \del^{M^*(\rho)}(f)&\coloneqq-(-1)^{|f|}f\del^M=(-1)^{|f|+1}f\del^M\\
a\cdot f&\coloneqq(-1)^{|a||f|}f(\rho(a)-).\end{align*} We are particularly interested in the relationship between $M^*(\id)$ and $M^*(\sigma)$. Furthermore, as $\Hom_k(\Lambda,k)\cong \shift^{-c} \Lambda$ as DG $\Lambda$-modules a direct calculation yields 
\begin{equation}\label{e:duals}
\RHom_{\Lambda}(M,\Lambda)\simeq \shift^c\Hom_k(M,k)=\shift^cM^*(\id).
\end{equation}
\end{chunk}

\begin{chunk}\label{twist}
Let $M$ be a DG $\Lambda$-module. Since $\Lambda$ has trivial differential we can twist the differential \text{and} $\Lambda$-action of $M$ to obtain an isomorphic DG $\Lambda$-module. We define $M_\tau$ to be the DG $\Lambda$-module whose underlying graded $k$-space is $M$ equipped with differential and $\Lambda$-action
\begin{align*}
    \del^{M_\tau}&\coloneqq -\del^M \\
    a\cdot m&\coloneqq (-1)^{|a|}am.
\end{align*}The map  $M\to M_\tau$ given by 
\[
m\mapsto (-1)^{|m|}m
\] is easily checked to be an isomorphism of DG $\Lambda$-modules.
\end{chunk}

\begin{proposition}\label{dualiso}
For any $M$ in  $\D(\Lambda)$, $M^*(\sigma)$ and $M^*(\id)_\tau$ have the same underlying graded $\Lambda$-module while their differentials are negatives of one another.
\end{proposition}
\begin{proof}
This follows directly from the definitions in \ref{dual} and \ref{twist}.
%
\end{proof}

\begin{chunk}\label{structure}
Let $M$ and $ N$ be left DG $\Lambda$-modules, then $M\otimes_k N$ is a left DG $\Lambda\otimes_k \Lambda$-module.
We define $\Hom_k(M,N)$ to be a left DG $\Lambda\otimes_k\Lambda$-module 
\begin{align*}
\del^{\Hom_k(M,N)}(f)\coloneqq\del^Nf-(-1)^{|f|}f\del^M\\
(a_1\otimes a_2)\cdot f\coloneqq(-1)^{|a_2||f|}a_1 f(\sigma(a_2)-).
\end{align*}
Hence, both $M\otimes_k N$ and $\Hom_k(M,N)$ inherit a DG $\Lambda$-module structure via $\Delta$. 
There is a natural morphism of DG $\Lambda$-modules
\[
\vp_{M,N}\colon N\otimes_k M^*(\sigma)\to  \Hom_k(M,N)\]
which is an isomorphism in $\D(\Lambda)$ when $\h(M)$ is a finite rank $k$-space (see  \cite[4.8]{AI3}). 
\end{chunk}

\begin{chunk}\label{split}
Let $M$ be a DG $\Lambda$-module. Consider the morphism of DG $\Lambda$-modules
$$
\pi_M\colon k\rightarrow \Hom_k(M,M)
$$
mapping $1$ to $\id_M$. 
$$
M\cong k\otimes_k M\xrightarrow{\pi_M\otimes M} \Hom_k(M,M)\otimes_k M.
$$  
For an object $M$ in $\D^f(\Lambda)$ the composition  \[ M\cong k\otimes_k M\xrightarrow{\pi_M\otimes M} \Hom_k(M,M)\otimes_k M\xra{\vp_{M,M}^{-1}\otimes M } M\otimes_k M^*(\sigma)\otimes_k M\] splits in $\D^f(\Lambda)$; this is
essentially the same argument from  the classical case  (cf. \cite[3.1.10]{benson}). 

%
\end{chunk}

\begin{proof}[Proof of Theorem \ref{closed}]
As  $\thick_{\D(\Lambda)}(k)=\D^f(\Lambda)$,  it follows that  \[M\otimes_k M^*(\sigma)\otimes_k M\text{ is an object of } \thick_{\D(\Lambda)}(M^*(\sigma)).\] Hence, \ref{split} implies that  $M$ is in $\thick_{\D(\Lambda)}(M^*(\sigma))$, as well.
Since $(M^*(\sigma))^*(\sigma)\cong M$, then by symmetry we have  shown 
\begin{equation}
    \thick_{\D(\Lambda)}(M)=\thick_{\D(\Lambda)} \left(M^*(\sigma)\right).
    \label{e:thick}
\end{equation}

Now we only need to observe that 
\begin{align*}
    \thick_{\D(\Lambda)}(M)&=\thick_{\D(\Lambda)} (M^*(\sigma)) \\
    &= \thick_{\D(\Lambda)} (M^*(\sigma),-\del^{M^*(\sigma)}) \\
    &=\thick_{\D(\Lambda)}(M^*(\id)) \\
    &=\thick_{\D(\Lambda)}\left(\RHom_\Lambda(M,\Lambda)\right)
\end{align*}
where the first equality holds by (\ref{e:thick}), the second equality is Lemma \ref{l:neg}, the third equality 
is from  \ref{twist} and Proposition \ref{dualiso}, and the last equality holds by \ref{dual}(\ref{e:duals}). Therefore, we have justified the theorem. 
%
%
\end{proof}



As an application of the theory above (in particular, Theorem \ref{closed}) we now present a second proof of Theorem \ref{cor1}.

\begin{proof}[Second proof of Theorem \ref{cor1}]
First, if $R$ is complete intersection we show that for any $M$ in $\D^f(R)$, \begin{equation}\thick_{\D(R)}\left(M\otimes_R K^R\right)=\thick_{\D(R)}\left(\RHom_R(M,R)\otimes_R K^R\right).\label{eq}\end{equation}

Now observe that Theorem \ref{closed} shows 
\[
\thick_{\D(\Lambda)}\left(\mathsf{j}\mathsf{t} M\right)=\thick_{\D(\Lambda)}\left(\RHom_\Lambda(\mathsf{j}\mathsf{t} M,\Lambda)\right) 
\] where $\mathsf{j}$ and $\mathsf{t}$ are the  functors  introduced in \ref{koszulci}.
Therefore,
\[
\thick_{\D(K^R)} (\mathsf{t} M)=\thick_{\D(K^R)}\left(\RHom_{K^R}(\mathsf{t} M,K^R)\right)
\]
and since  \[
\RHom_{K^{R}}(\mathsf{t} M,K^{R})\simeq \mathsf{t} \RHom_R(M,R)
\]
it follows that
\begin{equation*}\label{eq1}
\thick_{\D(K^{R})}(\mathsf{t}M) =\thick_{\D(K^{R})}\left(\mathsf{t}\RHom_R(M,R)\right).
\end{equation*} So restricting scalars along $R\to K^R$ finishes the proof of (\ref{eq}) in the case that  $R$ is complete intersection.

Now we return to the general setting; namely, assume that $R$ is locally complete intersection. Let $\p$ be  a prime ideal of $R$, then by assumption $R_\p$ is complete intersection and 
\begin{align*}
    \thick_{\D(R_\p)}\left(M(\p)\right)&= \thick_{\D(R_\p)}\left(\RHom_{R_\p}(M_\p,R_\p)\otimes_{R_\p}K^{R_\p}\right)\\
    &=\thick_{\D(R_\p)}\left(\RHom_{R}(M,R)(\p)\right);
\end{align*}
the first equality is from the already established equality (\ref{eq}) and the second equality 
 is immediate from the isomorphism $$\RHom_{R_\p}(M_\p,R_\p)\otimes_{R_\p}K^{R_\p}\simeq \RHom_R(M,R)\otimes_R K^{R_\p}.$$ Since these equalities of thick subcategories hold for each prime ideal $\p$ of $R$, an application of the local-to-global principle in \ref{lg}   establishes the desired result. 
\end{proof}


\section{Symmetry of Complexity}\label{s:complexity}

 In this section, we offer two proofs of Theorem \ref{symcx}; both methods indicate how the symmetry of complexity over  complete intersections follow from studying properties of exterior algebras. Theorem \ref{symcx} was originally shown in  \cite[5.7]{AB} using support varieties and the use of intermediate hypersurfaces. A second proof was given by the second author in \cite[4.3.1]{Po} by studying the cohomological support of certain DG modules over a  graded commutative ring of finite global dimension. 

Throughout this section  $(R,\m,k)$ is a commutative noetherian local ring. Also, we let $(-)^\vee$ denote the functor $\RHom_R(-,R).$
\begin{chunk}\label{cx}Let $M$ and  $N$ be  in $\D^f(R)$. The \emph{complexity of the  pair $(M,N)$},  denoted $\cx_R(M,N)$, is the least non-negative integer $d\in \mathbb N$ such that \[\dim_k(\Ext^n_R(M,N)\otimes_R k)\leq an^{d-1}\] for all $n\gg 0$ and some $a\in\mathbb R$. That is, $\cx_R(M,N)$ measures the polynomial rate of growth of the minimal number of generators of $\Ext_R^n(M,N)$.
\end{chunk}

\begin{theorem}\label{symcx}
Let $R$ be  complete intersection. For  each pair of objects $M$ and $N$ in $\D^f(R)$,  $\cx_R(M,N)=\cx_R(N,M).$
\end{theorem}

\begin{remark}\label{r1}
The first proof of Theorem \ref{symcx} is a straightforward application  of Theorem \ref{cor1} and two standard facts: 
\begin{enumerate}
    \item when $R$ is Gorenstein there is a natural isomorphism of graded $R$-modules
    \[\Ext_R(M,N)\cong \Ext_R(N^\vee,M^\vee)\] for each $M$ and $N$ in $\D^f(R)$; 
    \item  when $R$ is complete intersection \[\cx_R(X,Y)\leq \cx_R(M,N)\] whenever $X$ is in $\thick_{\D(R)}(M)$ and $Y$ is in $\thick_{\D(R)}(N)$ (see \cite[5.6]{AB} or \cite[4.2.5 \& 4.2.9]{Po}). 
\end{enumerate}
\end{remark}

\begin{proof}[First proof of Theorem \ref{symcx}]
From Remark \ref{r1}(1) it follows easily that  
$$
\cx_R(M,N)=\cx_R(N^\vee,M^\vee).
$$
Using Theorem \ref{cor1}, we obtain $M^\vee$ is an object of $\thick_{\D(R)}(M)$ and $N^\vee$ is an object of   $\thick_{\D(R)}(N)$. So  Remark \ref{r1}(2) establishes the inequality 
$$
\cx_R(M,N)=\cx_R(N^\vee,M^\vee)\leq \cx_R(N,M)
.$$
By symmetry the result follows.
\end{proof}

\begin{remark}\label{rcx}
As discussed in the introduction, the theme of \cite{AI3} is to deduce homological results over $R$, when  $R$ is complete intersection, using the bridge 
\[\mathsf{jt}\colon\D^f(R)\to\D^f(K^R)\xra{\equiv} \D^f(\Lambda)\]from \ref{koszulci}. However, in  \cite[8.9]{AI3} it is remarked that the authors did not see how to deduce Theorem \ref{symcx} from studying this bridge. The second proof of Theorem \ref{symcx}, given below, shows that  one can in fact arrive at the symmetry of complexity over the complete intersection $R$ as a direct consequence of the symmetry of complexity over  $\Lambda.$ 
\end{remark}

\begin{proof}[Second proof of Theorem \ref{symcx}]
Assume $R$ is complete intersection. 

Observe that 
\[
\Ext_{R}(M,\mathsf{t}N)\cong \Ext_{K^R}(\mathsf{t} M, \mathsf{t} N )
\cong \Ext_\Lambda(\mathsf{jt}M, \mathsf{jt}N)
\] where the first isomorphism is adjunction and the second one uses that  $\mathsf{j}$ is an equivalence. Also, by \cite[4.2.7]{Pol}
\[\cx_R(M,N)=\cx_R(M,\mathsf{t}N),\]
and so combining  this with the isomorphisms above we have \begin{equation}\label{e5}\cx_R(M,N)=\cx_\Lambda(\mathsf{jt}M, \mathsf{jt}N).\end{equation} 
Similarily, it follows that 
\begin{equation}\label{e8}\cx_R(N,M)=\cx_\Lambda(\mathsf{jt}N, \mathsf{jt}M).\end{equation}
Note \cite[5.3]{AI3} established 
\[
\cx_\Lambda(\mathsf{jt} M, \mathsf{jt}N)=\cx_\Lambda(\mathsf{jt}N, \mathsf{jt}M);
\]  this equality along with the ones in  (\ref{e5}) and (\ref{e8})  establish $\cx_R(M,N)=\cx_R(N,M)$, as claimed. 
\end{proof}


\bibliographystyle{amsplain}
\bibliography{ref}

\end{document}